\pdfoutput=1
\RequirePackage{ifpdf}
\ifpdf 
\documentclass[pdftex]{sigma}
\else
\documentclass{sigma}
\fi

\usepackage[all]{xy}

\numberwithin{equation}{section}

\newtheorem{Theorem}{Theorem}[section]
\newtheorem*{Theorem*}{Theorem}
\newtheorem{Corollary}[Theorem]{Corollary}
\newtheorem{Lemma}[Theorem]{Lemma}
\newtheorem{Claim}[Theorem]{Claim}
\newtheorem{Conjecture}[Theorem]{Conjecture}
\newtheorem{Proposition}[Theorem]{Proposition}
 { \theoremstyle{definition}

\newtheorem{Example}[Theorem]{Example}
\newtheorem{Remark}[Theorem]{Remark} }

\newcommand{\rk}{\operatorname{rk}}	

\newcommand{\sym}{\textnormal{Sym}}

\begin{document}

\allowdisplaybreaks

\renewcommand{\thefootnote}{}

\renewcommand{\PaperNumber}{061}

\FirstPageHeading

\ShortArticleName{Big and Nef Tautological Vector Bundles over the Hilbert Scheme of Points}

\ArticleName{Big and Nef Tautological Vector Bundles\\ over the Hilbert Scheme of Points\footnote{This paper is a~contribution to the Special Issue on Enumerative and Gauge-Theoretic Invariants in honor of Lothar G\"ottsche on the occasion of his 60th birthday. The~full collection is available at \href{https://www.emis.de/journals/SIGMA/Gottsche.html}{https://www.emis.de/journals/SIGMA/Gottsche.html}}}

\Author{Dragos OPREA}

\AuthorNameForHeading{D.~Oprea}

\Address{Department of Mathematics, University of California San Diego,\\ 9500 Gilman Drive, La Jolla, CA, USA}
\Email{\href{mailto:doprea@math.ucsd.edu}{doprea@math.ucsd.edu}}
\URLaddress{\url{http://math.ucsd.edu/~doprea/}}

\ArticleDates{Received January 31, 2022, in final form July 31, 2022; Published online August 12, 2022}

\Abstract{We study tautological vector bundles over the Hilbert scheme of points on surfaces. For each $K$-trivial surface, we write down a simple criterion ensuring that the tautological bundles are big and nef, and illustrate it by examples. In the $K3$ case, we extend recent constructions and results of Bini, Boissi\`ere and Flamini from the Hilbert scheme of~2 and~3 points to an arbitrary number of points. Among the $K$-trivial surfaces, the case of Enriques surfaces is the most involved. Our techniques apply to other smooth projective surfaces, including blowups of $K3$s and minimal surfaces of general type, as well as to the punctual Quot schemes of curves.}


\Classification{14C05; 14D20; 14C17}

\renewcommand{\thefootnote}{\arabic{footnote}}
\setcounter{footnote}{0} \vspace{-1ex}

\section{Introduction}Let $X$ be a smooth projective surface, and let $X^{[k]}$ denote the Hilbert scheme of $k$ points on $X$. Each vector bundle $F\to X$ of rank $r$ yields a tautological vector bundle $F^{[k]}\to X^{[k]}$ of rank $rk$ given by \vspace{-.5ex}
\begin{gather*}
F^{[k]}=p_{\star} (q^{\star} F\otimes \mathcal O_{\mathcal Z}).
\end{gather*}
Here, $p$, $q$ are the natural projections from $X^{[k]}\times X$, and $\mathcal Z\subset X^{[k]}\times X$ denotes the universal subscheme.

The literature surrounding the geometry of the tautological bundles is vast. Likewise, many notions of positivity for vector bundles have been studied in algebraic and complex differential geometry. Merging these two themes, it is natural to investigate the positivity properties of the tautological bundles.

In this note, we address the question whether the bundles $F^{[k]}\to X^{[k]}$ are big and nef.
To our knowledge, for $K3$s, this has been considered for the first time in the recent article \cite{BBF}, alongside the stability and bigness of twists of the tangent bundle of $X^{[k]}$. Specifically, if $X$ a $K3$ surface of Picard rank $1$, and the number of points is $k=2, 3$, it is shown in \cite{BBF} that $F^{[k]}$ is big and nef when $F$ is either\vspace{-.5ex}
\begin{itemize}\itemsep=-1pt
\item [$(a)$] a positive line bundle,
\item [$(b)$] a twist of a Lazarsfeld--Mukai bundle (for suitable numerics),
\item [$(c)$] a twist of an Ulrich bundle.
\end{itemize}

Recall that a vector bundle $V$ over a scheme $Y$ is said to be big and nef if the line bundle $\mathcal O_{\mathbb P(V)}(1)\to \mathbb P(V)$ is big and nef, where $\mathbb P(V)$ denotes the projective bundle of one dimensional quotients. A discussion of big and nef vector bundles can be found in \cite[Chapters~6 and~7]{L-II}.
A useful well-known characterization occurs when $V\to Y$ is globally generated. In this case, if the top Segre class
\begin{equation}
\label{pos}(-1)^{\dim Y} \int_{Y} s(V)>0
\end{equation}
it follows that $V\to Y$ is big and nef.

\subsection{Results} The original motivation for our work was provided by the recent results of \cite{BBF}, which we extend in several directions.
\begin{itemize}\itemsep=0pt
\item [$(i)$] For $K3$s, we allow for arbitrary number of points $k$, and derive a condition that ensures $F^{[k]}$ is big and nef, see Theorem~\ref{tt}. We apply this theorem to obtain analogues of examples $(a)$--$(c)$ above for any $k$.
\item [$(ii)$] We allow for arbitrary $K$-trivial surfaces. The case of Enriques surfaces is the most difficult, and we only have results in odd rank, see Theorem~\ref{tt2}, as well as a conjectural bound in general.
\item [$(iii)$] We consider other smooth projective surfaces, including blowups of $K3$s in Theorem~\ref{blok}, and minimal surfaces of general type, in rank~$1$, in Theorem~\ref{gtyp}. The latter theorem is the most involved result we prove here, requiring a more detailed analysis than for other geometries.
\item [$(iv)$] We show how the same ideas yield similar results over the punctual Quot schemes of curves, see Theorem~\ref{quotpun}.
\end{itemize}

Compared to \cite{BBF}, the new ingredient is the closed form calculation of the Segre integrals in \cite{MOP2, MOP1, MOP, OP2021}. The formulas are explicit, and the goal here is to show how to apply them to derive geometric positivity results. This is not always immediate, and the arguments require several different ideas. We thus believe it is worthwhile to record the outcome. We also illustrate our calculations by a few geometric examples.

\subsection{Applications} By \cite[Proposition 1.4]{Y}, taking determinants of big and nef vector bundles yields big and nef divisors. There are several results in the literature concerning the positivity of the determinants $\det F^{[k]}$, see for instance \cite{BS, CG} with regards to very ampleness when $F$ has rank $1$, for arbitrary surfaces. In general, nef divisors over the Hilbert scheme of $K3$s were studied in \cite[Section~10]{BM}. Over other surfaces, related results can be found in \cite{ABCH, BC, BHL+, Ko, LQZ, MM, N, QT, YY}, among others. The nef cones of divisors of the punctual Quot schemes of curves of genus $0$ and $1$ were determined in \cite{GS, St}.

Whenever $F^{[k]}\to X^{[k]}$ is a big and nef bundle, for every $m_1, \dots, m_h\geq 0$, Demailly vanishing gives\footnote{For an ample vector bundle $V\to Y$, Demailly's vanishing theorem \cite{De} states
\begin{gather*}
H^i\big(Y, \omega_Y\otimes \sym^{m_1}V\otimes \dots \otimes \sym^{m_h}V\otimes (\det V)^{\otimes h}\big)=0, \qquad m_1, \dots, m_h\geq 0, \qquad h>0, \qquad i> 0.
\end{gather*}
This is derived in \cite[Theorem 7.3.14]{L-II} from Griffiths vanishing. The same argument applies to big and nef vector bundles; \cite[Example 7.3.3]{L-II} notes that Griffiths vanishing holds in this context.}
\begin{gather*}
H^i\bigl(X^{[k]}, \omega_{X^{[k]}} \otimes \sym^{m_1}F^{[k]}\otimes \dots \otimes \sym^{m_h} F^{[k]}\otimes \bigl(\det F^{[k]}\bigr)^{\otimes h}\bigr)=0, \qquad i>0.
\end{gather*}
For instance, using Theorem \ref{tt} or Corollary \ref{c4}, if $L\to X$ is an ample line bundle on a $K3$ surface $X$ of Picard rank $1$, with $\chi(L)\geq 3k$, we have
\begin{gather*}
H^i\bigl(X^{[k]}, \sym^{m_1}L^{[k]}\otimes \dots \otimes \sym^{m_h} L^{[k]}\otimes \bigl(\det L^{[k]}\bigr)^{\otimes h}\bigr)=0, \qquad i>0.
\end{gather*}
Analogous statements hold in all geometric situations covered by items $(i)$--$(iv)$ above. Cohomology with values in the tautological bundles and their representations was studied for instance in \cite{A, Da, EGL, Sc, Z}, but the vanishings results above are new.

To further understand the cohomology, the next step would be to compute the holomorphic Euler characteristics, that is to find the series
 \begin{gather*}
 \mathsf Z_{X, F}^{m_1, \dots, m_h}=\sum_{k=0}^{\infty} q^k \chi \bigl(X^{[k]}, \omega_{X^{[k]}} \otimes \sym^{m_1}F^{[k]}\otimes \dots \otimes \sym^{m_h} F^{[k]}\otimes \bigl(\det F^{[k]}\bigr)^{\otimes h}\bigr).
 \end{gather*}
This is a difficult but interesting question. We expect that the answer is given by {\it algebraic} functions. The simplest case $m_1=\dots=m_h=0$ corresponds to the Verlinde series determined in \cite[Theorem~5.3]{EGL} for $K$-trivial surfaces, and conjectured for small $h$ for all surfaces in \cite[Section 1.6]{MOP}. (After the writing was completed, we learned about the recent announcement~\cite{G} regarding expressions for the Verlinde series for all surfaces and all values of~$h$.)

Similar vanishing statements can be made over the punctual Quot schemes of curves.

\subsection{Plan of the paper} The case of $K3$ surfaces is the simplest and is discussed first, see Section~\ref{seck3}. To illustrate the results of Section~\ref{seck3}, in Section~\ref{secex} we extend the constructions in~\cite{BBF} to arbitrary number of points. Other $K$-trivial surfaces, and in particular Enriques surfaces, are considered in Section~\ref{secktriv}. Other geometries, specifically $K3$ blowups and minimal surfaces of general type, are studied in Section~\ref{secother}. Sections~\ref{secktriv} and~\ref{secother} are the most involved. Finally, Section \ref{seccur} concerns the punctual Quot scheme of curves.

\section[K3 surfaces]{$\boldsymbol{K3}$ surfaces} \label{seck3} Let $X$ be a smooth projective surface. The bundle $F^{[k]}\to X^{[k]}$ is globally generated, and therefore nef, provided $F\to X$ is $(k-1)$-very ample. By definition, $(k-1)$-very ampleness is the requirement that the natural map
\begin{gather*}
H^0(X, F)\to H^0(X, F\otimes \mathcal O_{\zeta})
\end{gather*}
is surjective for all zero-dimensional subschemes $\zeta$ of $X$ of length $k$. Thus, via \eqref{pos}, if $F$ is $(k-1)$-very ample and
\begin{gather*}
\int_{X^{[k]}} s\big(F^{[k]}\big)>0,
\end{gather*}
then $F^{[k]}\to X^{[k]}$ is big and nef. This is explained for instance in \cite[Propositions~2.4 and~4.5]{BBF}.

Let $(X, H)$ be a polarized $K3$ surface, and let $r= \operatorname{rank}F$. Central for our argument is the following structural expression for the Segre integrals established in \cite{MOP}: \begin{equation}\label{eq1}
\sum_{k=0}^{\infty} z^k \int_{X^{[k]}} s \big(F^{[k]}\big) = A_0(z)^{c_2(F)} A_1(z)^{c_1^2(F)} A_2(z).
\end{equation}
{\samepage
The series $A_0$, $A_1$, and $A_2$ are given by explicit algebraic functions
\begin{gather}
A_0(z)= (1+(1+r)t)^{-r-1} (1 + (2+r) t)^{r}, \nonumber
\\
A_1 (z)= (1+(1+r)t)^{\frac{r}{2}} (1 + (2+r) t)^{-\frac{r-1}{2}},\nonumber
\\
A_2 (z)= (1+(1+r)t)^{r^2+2r} (1 + (2+r)t)^{-r^2+1} (1 + (1+r)(2+r)t)^{-1},\label{aaa}
\end{gather}
for the change of variables
\begin{equation*}
z = t (1+(1+r)t)^{1+r}.
\end{equation*}}
We point out that in \cite{MOP}, $r$ stands for $\operatorname{rank}(F) + 1$, while for us $r=\operatorname{rank}(F)$; the above expressions account for the different notational conventions. Related formulas over the moduli space of higher rank sheaves were proposed in \cite{GK} and were recently proven in \cite{Ob}.

For each rank $r$ vector bundle $F\to X$, we write $v=v(F)=\operatorname{ch} (F) \sqrt{\operatorname{Td}(X)}$ for its Mukai vector, and we set
\begin{gather*}
\chi=\chi(F),\qquad \delta=1+\frac{1}{2}\langle v, v\rangle.
\end{gather*}
Here $\langle\,,\, \rangle$ is the Mukai pairing given by
\begin{gather*}
\langle v, v\rangle =\int_{X} v_2^2-2v_0v_4 \qquad
\text{for vectors}\quad v=(v_0, v_2, v_4)\in H^{0}(X)\oplus H^2(X)\oplus H^4(X).
\end{gather*}
Moreover, $\delta$ equals the expected dimension of the moduli space of sheaves of type $v$.
We show
\begin{Theorem}\label{tt}
Assume $\chi \geq (r+2) k$ and $\delta \geq 0$.
If $F$ is $(k-1)$-very ample, then $F^{[k]}$ is a big and nef vector bundle over $X^{[k]}$.
\end{Theorem}

\begin{proof}
By the first paragraph of this section, it suffices to show that the integral of the Segre class $s_{2k}\big(F^{[k]}\big)$ is positive. To this end, we use an equivalent form of equation \eqref{eq1}, which can be found in \cite[p.~11]{MOP}.
Specifically, via a residue calculation, it was established there that
\begin{equation}\label{e1}
\int_{X^{[k]}} s_{2k}\big(F^{[k]}\big)=\operatorname{Coeff}_{ t^k} \bigl[(1+ (2+r)t)^{\delta} (1+(1+r)t)^{\chi -\delta - (r+1)k} \bigr].
\end{equation}
Expanding via the binomial theorem, equation~\eqref{e1} shows that the Segre integral is positive when the expression between brackets is a polynomial of degree at least $k$, which is the case for
\begin{gather*}
\delta\geq 0, \qquad \chi- \delta - (r+1)k\geq 0,\qquad \chi-(r+1)k\geq k.
\end{gather*}
However, since $\delta$ could be large, we seek better bounds. To this end, we change variables, setting \begin{gather*}
t=\frac{u}{1-(1+r)u}.
\end{gather*}
We rewrite \eqref{e1} as
\begin{align*}
\int_{X^{[k]}} s_{2k}\big(F^{[k]}\big)&=\operatorname{Res}_{t=0} (1+ (2+r)t)^{\delta} (1+(1+r)t)^{\chi -\delta - (r+1)k} \frac{\mathrm dt}{t^{k+1}}
\\
&= \operatorname{Res}_{u=0} (1+u)^{\delta} (1-(1+r)u)^{-\chi + (r+2) k-1} \frac{\mathrm du}{u^{k+1}}
\\
&= \operatorname{Coeff}_{u^k} (1+u)^{\delta} (1-(1+r)u)^{-\chi + (r+2) k-1}.
\end{align*}
Letting $a_i$ denote the coefficients of the term $(1+u)^{\delta}$, we have $a_i>0$ for $0\leq i\leq \delta$. Similarly, the coefficients of the second term are
\begin{gather*}
b_j=(1+r)^j (-1)^j \binom{-\chi +(r+2)k-1}{j}>0,
\end{gather*}
since $-\chi+(r+2)k-1<0$. Here, we use the standard definition of binomial numbers
\begin{gather*}
\binom{x}{j}=\frac{x(x-1)\cdots (x-j+1)}{j!}
\end{gather*}
for arbitrary $x$. Thus, the Segre integral equals $\sum a_i b_j$, the sum ranging over $i+j=k$, $0\leq i\leq \delta$, $j\geq 0$. The integral is positive since each term $a_ib_j>0$, and the sum is non-empty (it contains the term $a_0b_k$). The proof is complete.
\end{proof}

\begin{Remark}
The theorem is certainly not optimal in all cases, but it suffices for our purposes. To illustrate it, when $\delta=0$, equation \eqref{e1} yields the following result originally noted in \cite[Proposition 2.1]{MOP}:
\begin{equation}\label{f1}
\int_{X^{[k]}} s_{2k}\big(F^{[k]}\big)=(r+1)^k \binom{\chi-(r+1)k}{k}.
\end{equation}
In this case, the positivity of the Segre integral is guaranteed when $\chi\geq (r+2)k$, but also when $\chi<(r+1)k$ and $k$ even.

The vanishing of the Segre integrals~\eqref{f1} for exceptional bundles with $(r+1)k\leq \chi<(r+2)k$ played an important role in the proof of~(2.2) in~\cite{MOP}. The point of Theorem~\ref{tt} is that we can furthermore pin down the sign of the Segre integral for $\chi$ to the right of the above interval.
\end{Remark}

\begin{Remark}\label{remk}
In rank $1$, $(k-1)$-very ampleness of nef line bundles over $K3$ surfaces can be effectively studied using \cite[Theorem~2.1]{BS}. Specifically, if
$L$ is nef and $L^2>4k$, then either $L$ is $(k-1)$-very ample or else there exists an effective divisor $D$ such that $L-2D$ is $\mathbb Q$-effective, with
\begin{equation}\label{kva}
L.D-k\leq D^2<L.D/2<k.
\end{equation}
Furthermore, $D$ contains a subscheme $\zeta$ of length less or equal to $k$ such that
\begin{gather*}
H^0(L)\to H^0(L|_{\zeta})\quad\text{is not surjective.}
\end{gather*}
Over arbitrary smooth projective surfaces, a similar result ensures the $(k-1)$-very ampleness of the adjoint bundles $K_X+L$.

To our knowledge, an analogous criterion in higher rank is missing. We point out two constructions yielding $(k-1)$-very ample bundles over $K3$ surfaces:
\begin{itemize}\itemsep=0pt
\item [$(i)$] If $(X, H)$ satisfies $\text{Pic} (X)=\mathbb Z\langle H\rangle$ and $F$ is a $\mu_H$-stable vector bundle with $\det F=H$ and $\chi\geq (r+1)k+\delta$ then $F$ is $(k-1)$-very ample. This assertion follows by the proof of \cite[Proposition 2.2]{MOP}.
\item [$(ii)$] By \cite[Proposition 4.5]{BBF}, over any surface, twisting a globally generated vector bundle by a~$(k-1)$-very ample line bundle yields a $(k-1)$-very ample bundle (with large determinant).
\end{itemize}
For abelian surfaces, other constructions are possible via isogenies or extensions, see Section~\ref{ababab}.
\end{Remark}

\subsection[Big and nef tautological bundles over the Hilbert scheme of K3s]
{Big and nef tautological bundles over the Hilbert scheme of $\boldsymbol{K3}$s} \label{secex}

Theorem \ref{tt} applies to the three examples considered in Theorems~5.3, 5.5 and~5.7 and Corollaries~5.4, 5.6 and~5.8 of \cite{BBF}, and mentioned in items $(a)$--$(c)$ of the Introduction: line bundles and twists of Lazarsfeld--Mukai or Ulrich bundles. The goal here is to show how to extend the results in \cite{BBF} from $k\leq 3$ to any number of points.

Throughout this section, we assume $(X, H)$ is a $K3$ surface of Picard rank $1$, and $\text{Pic}(X)=\mathbb Z\langle H\rangle$. Let $H^2=2g-2$.

\begin{Corollary} \label{c4}Let $L_n=H^{\otimes n}$ for $n\geq 1$. Assume $g\geq 3k-1$. Then $(L_n)^{[k]}$ is big and nef over~$X^{[k]}$.
\end{Corollary}

\begin{proof}
Global generation, and thus nefness, is explained in \cite[Theorem 5.3]{BBF}. To prove bigness, as also noted in \cite{BBF}, it suffices to establish the positivity of the Segre integral
\begin{gather*}
\int_{X^{[k]}} s_{2k} \big((L_n)^{[k]}\big)>0.
\end{gather*}
By formula \eqref{f1}, we have
\begin{gather*}
\int_{X^{[k]}} s_{2k} \big((L_n)^{[k]}\big)=2^k \binom{\chi(L_n)-2k}{k},
\end{gather*}
which is positive provided
\begin{gather*}
\chi(L_n)\geq 3k\iff 2+n^2(g-1)\geq 3k.
\end{gather*}
The latter inequality is clear under our hypothesis.
\end{proof}

We next consider Lazarsfeld--Mukai bundles and their induced tautological bundles over $X^{[k]}$. This geometric situation corresponds to Theorem 5.5 and Corollary 5.6 in \cite{BBF}. There are several ways of formulating the result, but in keeping with \cite{BBF}, we prefer bounds which do not depend on the rank.

Recall that the Lazarsfeld--Mukai bundles are obtained as duals $E=K_{C, L}^{\vee}$ to kernels
\begin{gather*}
0\to K_{C, L}\to H^0(X, L)\otimes \mathcal O_X\to \iota_{\star} L\to 0.
\end{gather*}
Here $L\to C$ is a line bundle over a nonsingular curve $C\in |H|$, of degree $d$ and with $r\geq 2$ sections, such that $L$ and $\omega_C\otimes L^{\vee}$ are globally generated. It follows that $E$ is globally generated, with
\begin{gather*}
\rk E=r\geq 2, \qquad c_1(E)=H, \qquad c_2(E)=d\implies v(E)=(r, H, g-1-d+r).
\end{gather*}
We let $\rho=g-r(r-1+g-d)$ denote the Brill--Noether number.

\begin{Example}
Let $E$ be a globally generated Lazarsfeld--Mukai bundle as above. Assume that
\begin{gather*}
\rho\geq 0, \qquad g>2k-2>0, \qquad g> \frac{2}{5}(d+1).
\end{gather*}
Then $(E\otimes H)^{[k]}$ is big and nef.
\end{Example}

\begin{proof}
Let $F=E\otimes H$. The assumption $g>2k-2$ was used in \cite[Theorem~5.5]{BBF} to prove that $F=E\otimes H$ is $(k-1)$-very ample; this is based on the result cited in Remark~\ref{remk}$(ii)$. As~noted in \cite{BBF}, bigness follows once we verify that $\int_{X^{[k]}} s_{2k}\big(F^{[k]}\big)>0$. To this end, we check that the assumptions of Theorem~\ref{tt} hold true. A simple calculation yields
\begin{gather*}
\chi(F)=g(r+3) -d +r-3,\qquad \delta(F)=1+\frac{1}{2} \langle v(F), v(F)\rangle=\rho.
\end{gather*}
The inequality $\chi(F)\geq (r+2)k$ is satisfied.
Indeed, by hypothesis $g\geq 2k-1$ and $g\geq \frac{2d+3}{5}$, hence averaging we have $g\geq k+\frac{d-1}{5}$. Then
\begin{align*}
\chi(F)-(r+2)k&=g(r+3)-d+r-3-(r+2)k
\\
&\geq \biggl(k+\frac{d-1}{5}\biggr)(r+3)-d+r-3-(r+2)k
\\
&=\frac{d+4}{5}(r-2)+(k-2)\geq 0,
\end{align*}
since $r, k\geq 2$.
\end{proof}

\begin{Example} The (untwisted) Lazarsfeld--Mukai bundles also yield big and nef vector bundles over $X^{[k]}$, under more restrictive assumptions. Take $r=2$ for simplicity. Assume
\begin{gather*}
2d-2\geq g> 2k-3+ \frac{3}{2}d,\qquad \text{which implies} \quad \chi(E)\geq 4k+\rho, \quad \rho\geq 0.
\end{gather*}
This is a bit stronger than what is needed, but it ensures $\chi\geq 4k$ and $\chi\geq 3k+\rho$ simultaneously. It is well known that $E$ is $\mu_H$-stable when $\text{Pic}(X)=\mathbb Z\langle H\rangle$. (Reason: any destabilizing quotient has rank $1$, slope $\leq 0$, and is globally generated since $E$ is. Hence the quotient is trivial, and thus $c_2(E)=d=0$, a contradiction.) Since $c_1(E)=H$, it follows $E$ is $(k-1)$-very ample by Remark \ref{remk}$(i)$. By Theorem \ref{tt}, we have that $E^{[k]}$ is big and nef.
\end{Example}

Finally, we turn to Ulrich bundles considered in Theorem 5.7 and Corollary 5.8 in \cite{BBF}. We~write $H^2=2h$, so that $g=h+1$. Recall that a bundle $E$ over $(X, H)$ is said to be Ulrich if
\begin{gather*}
H^{\star}(X, E(-H))=H^\star(X, E(-2H))=0.
\end{gather*}
Such bundles always exist for $K3$ surfaces of Picard rank $1$ by \cite[Theorem 1.5]{AFO}, and they have numerics
\begin{gather*}
\rk E=2a, \qquad c_1(E)=3aH,\qquad c_2(E)=9a^2h-4a(h-1).
\end{gather*}

\begin{Example}
Assume $h>2k-3>0$. Consider an Ulrich bundle $E$ as above. Then $(E\otimes H)^{[k]}$ is big and nef on the Hilbert scheme $X^{[k]}$.
\end{Example}

\begin{proof} Letting $F=E\otimes H$, we compute
\begin{gather*}
\chi(F)=12ah, \qquad \delta(F)=1+a^2h+4a^2.
\end{gather*}
As noted in \cite[Theorem 5.6]{BBF}, the bundle $F=E\otimes H$ is $(k-1)$-very ample if $h>2k-3$; this uses the statement cited in Remark 3(ii). To prove bigness, it remains to verify the positivity of the top Segre integral. By Theorem \ref{tt}, we check that
\begin{gather*}
\chi(F)\geq (r+2)k\iff 12ah\geq (2a+2)k.
\end{gather*}
This is clear since $h> 2k-3>0$.
\end{proof}

\begin{Example}
We can extend the result to Ulrich bundles $E$ over the polarized $K3$ surface $(X, mH)$ for all $m\geq 1$. By \cite[Proposition 4.4]{CNY}, the Mukai vectors of such Ulrich bundles are of the form
\begin{gather*}
v=\biggl(r, \frac{3rm}{2}H, h\big(2m^2r\big)-r\biggr).
\end{gather*}
By the same arguments, $(E\otimes H)^{[k]}$ is big and nef for $h>2k-3>0$.
\end{Example}

\section[Other K-trivial surfaces]
{Other $\boldsymbol K$-trivial surfaces} \label{secktriv}

\subsection{Abelian and bielliptic surfaces}\label{ababab}
The formulas in \cite{MOP} can be used to treat the case of abelian or bielliptic surfaces. For vector bundles $F\to X$, we set\vspace{-.5ex}
\begin{gather*}
r=\operatorname{rank}F,\qquad \chi=\chi(F),\qquad v=\operatorname{ch}(F), \qquad \delta=\frac{1}{2}\langle v, v\rangle.
\end{gather*}
We show
\begin{Theorem}
For a $(k-1)$-very ample vector bundle $F\to X$,
the tautological bundle $F^{[k]}\to X^{[k]}$ is big and nef provided that
$\chi\geq (r+2)k$ and~$\delta\geq 0$.
\end{Theorem}

\begin{proof}
We follow the same steps as for $K3$ surfaces, but a few numerical changes are necessary. First, it was noted on \cite[p.~19]{MOP} that the analogue of equation \eqref{e1} for abelian or bielliptic surfaces takes the form
\begin{gather*}
\int_{X^{[k]}} s_{2k} \big(F^{[k]}\big)= \operatorname{Coeff}_{t^k}
\bigl[(1+ (2+r)t)^{\delta} (1+(1+r)t)^{\chi-\delta -(r+1)k-1} (1+(1+r)(2+r)t)\bigr].
\end{gather*}
Next, the change of variables
\begin{gather*}
t=\frac{u}{1-(1+r)u}
\end{gather*}
turns the above expression into
\begin{gather*}
\int_{X^{[k]}} s_{2k} \big(F^{[k]}\big)=\operatorname{Coeff}_{u^k}
\bigl[(1+u)^{\delta} (1-(1+r)u)^{-\chi+(r+2)k-1} \big(1+(1+r)^2u\big)\bigr].
\end{gather*}
The proof is completed by the same argument as in Theorem \ref{tt}, this time letting $a_i$ be the~coefficients of $(1+u)^{\delta}\big(1+(1+r)^2u\big)$ and letting $b_j$ be the~coefficients of $(1-(1+r)u)^{-\chi+(r+2)k-1}$.
\end{proof}

\begin{Corollary}\label{c10}
If $L$ is $(k-1)$-very ample line bundle over an abelian or bielliptic surface and $L^2\geq 6k$, then $L^{[k]}$ is big and nef.
\end{Corollary}

In the following examples, we assume $X$ is an abelian surface with Picard rank $1$, with N\'eron--Severi ample generator $H$.

\begin{Example}
Assume $H^2\geq 6k$. For $H^2>4k$, the line bundle $L_n=H^{\otimes n}$ is $(k-1)$-very ample for all $n\geq 1$. This is an immediate consequence of \cite[Theorem 2.1]{BS}. Indeed, using that the Picard rank is $1$, the inequality \eqref{kva} is impossible. For $H^2\geq 6k$, we also have $\chi(L_n)\geq 3k$. It~follows from Corollary \ref{c10} that $(L_n)^{[k]}$ is big and nef.
\end{Example}

In higher rank, just as for $K3$s, the reader can consider twists of Ulrich and Lazarfeld--Mukai bundles. Here, taking advantage of the abelian surface geometry, we discuss simple semihomogeneous bundles, and twists of unipotent and homogeneous bundles.
\begin{Example}
 Assume that $H$ is a principal polarization. Let $(a, b)=1$ be coprime positive integers with
 \begin{gather*}
 b>a^2k.
 \end{gather*}
 By \cite[Remark 7.13]{M}, there exist simple semihomogeneous vector bundles $W\to X$ with
 {\samepage\begin{gather*}
 \rk W=a^2, \qquad \mu(W)=\frac{bH}{a}\in NS(X)\otimes \mathbb Q, \qquad \chi(W)=b^2.
 \end{gather*}
 We claim $W^{[k]}$ is big and nef.}

We note first that $W$ is $(k-1)$-very ample. Indeed, it is shown in \cite[Theorem 5.8] {M} that
\begin{gather*}
W=f_{\star} L
\end{gather*}
for some line bundle $L\to Y$, where $f\colon Y\to X$ is an isogeny of degree $a^2$. It was remarked in \cite[Section 2.4]{Op} that $W$ has no higher cohomology for $b>0$. Consequently, the same is true about $L$. Since
\begin{gather*}
h^0(L)=h^0(W)=\chi(W)=b^2,
\end{gather*}
we conclude that $L$ is effective and $L^2=2b^2$.

We claim $L$ is $\big(a^2k-1\big)$-very ample for $b>a^2k$. This follows again by \cite[Theorem 2.1]{BS}. Indeed, the Picard rank is invariant under isogenies \cite[Proposition~3.2]{BL}, hence $Y$ has Picard rank $1$ since $X$ does. If $M$ is the ample N\'eron--Severi generator, write $L\equiv_{\text{num}} M^{\ell}$, for $\ell> 0$. We have
\begin{gather*}
2b^2=L^2=\ell^2 M^2\geq 2\ell^2\implies 0< \ell \leq b.
\end{gather*}
For any effective divisor $D\neq 0$, we have
\begin{gather*}
L.D= \ell M.D\geq \ell M.M=\frac{L^2}{\ell}=\frac{2b^2}{\ell}\geq 2b\geq 2a^2k.
\end{gather*}
This violates \eqref{kva}. Since $L$ is nef and $L^2=2b^2>4a^2k$, it follows that $L$ is $\big(a^2k-1\big)$-very ample.

To check $(k-1)$-very ampleness for $W$, note that if $\zeta$ is a subscheme of $X$ of length $k$, then \begin{gather*}
H^0(W)\to H^0(W\otimes\mathcal O_{\zeta}) \text{ surjective }\iff H^0(L)\to H^0(L\otimes \mathcal O_{f^{\star} \zeta}) \text { surjective}.
\end{gather*}
The latter is true since $L$ is $\big(a^2k-1\big)$-very ample and $f^{\star}\zeta$ has length $a^2k$.

Finally, the inequality
\begin{gather*}
\chi(W)\geq (r+2) k\iff b^2\geq \big(a^2+2\big)k
\end{gather*}
is certainly true when $b>a^2k$.
\end{Example}

\begin{Example} Assume $H^2>4k$. Let $E$ be a unipotent bundle of rank $r\geq 2$, that is, $E$~admits a filtration whose successive quotients are the trivial line bundle \cite[Definition~4.5]{M}. Let $F=E\otimes H$. If $H^2> 4k$ then $H$ is $(k-1)$-very ample by \cite[Theorem 2.1]{BS}, see \eqref{kva}. It follows that~$F$ is also $(k-1)$-very ample. Indeed, $F$ is obtained as an iterated extension of $H$. At each step, we inductively show that the extension is $(k-1)$-very ample without higher cohomology, by examining the relevant short exact sequences. The filtration of $F$ also gives
\begin{gather*}
\chi(F)=r\chi(H)> 2rk\geq (r+2)k.
\end{gather*}
Thus $(E\otimes H)^{[k]}$ is big and nef.

Let $E$ be homogeneous bundle of rank $r\geq 2$, that is, $E$ is invariant by translations on $X$. By \cite[Theorem 4.17]{M}, we can write
\begin{gather*}
E=\bigoplus_i U_i\otimes P_i,
\end{gather*}
with $U_i$ unipotent, and $P_i$ a line bundle of degree $0$. Repeating the argument above for each summand, we show first that $E\otimes H$ is $(k-1)$-very ample, and then $(E\otimes H)^{[k]}$ is big and nef.
\end{Example}

For bielliptic surfaces, we leave specific examples to the reader, mentioning only that $(k-1)$-very ampleness of line bundles is studied in \cite{MP}.

\subsection{Enriques surfaces} By contrast, the case of Enriques surfaces requires more care, due to the shape of the algebraic functions giving the Segre integrals. As before, we write $v(F)=\operatorname{ch}(F) \sqrt{\operatorname{Td}(X)}$ for the Mukai vector, and set
\begin{gather*}
\delta=\frac{1}{2}+\frac{1}{2}\langle v, v\rangle \implies \delta=rc_2-\frac{r-1}{2}c_1^2-\frac{r^2-1}{2}.
\end{gather*}
Note that $\delta$ is an integer if and only if the rank $r$ is odd. In this case, we prove:

\begin{Theorem}\label{tt2}
Let $F$ be a $(k-1)$-very ample bundle of odd rank $r$ on an Enriques surface, with
\begin{gather*}
\chi\geq 2k (r+1), \qquad \delta\geq 0.
\end{gather*}
Then $F^{[k]}$ is big and nef.
\end{Theorem}
Computer experiments show that the bound $\chi\geq (r+2)k$ imposed for the other $K$-trivial surfaces is insufficient here. Nonetheless, we have the following

\begin{Conjecture}
For all ranks, odd or even, Theorem $\ref{tt2}$ holds under the weaker assumption
\begin{gather*}
\chi\geq \biggl(\frac{5r}{4}+2\biggr)k,\quad \delta\geq 0.
\end{gather*}
\end{Conjecture}

\begin{proof}[Proof of Theorem \ref{tt2}] Since $F$ is $(k-1)$-very ample, $F^{[k]}$ is globally generated, hence nef. To establish bigness, it remains to show that the degree of the top Segre class $s_{2k}\big(F^{[k]}\big)>0$.

By \cite[Theorem 1]{MOP}, the Enriques analogue of \eqref{eq1} takes the form \begin{equation}\label{eqx}
\sum_{k=0}^{\infty} z^k \int_{X^{[k]}} s \big(F^{[k]}\big) = A_0(z)^{c_2(F)} A_1(z)^{c_1(F)^2} A_2(z)^{\frac{1}{2}}
\end{equation}
for exactly the same universal functions $A_0$, $A_1$, $A_2$ which appear for $K3$ surfaces. Using \eqref{eqx} and following the reasoning in \cite[p.~11]{MOP}, with the modified numerics, we obtain an expression for the top Segre integral
\begin{gather}\label{pro}
\int_{X^{[k]}} s_{2k}\big(F^{[k]}\big)\nonumber
\\ \qquad
{}=\operatorname{Coeff}_{t^k} \big[(1+(1+r)t)^{\chi-\delta-k(r+1) - \frac{1}{2}} (1+(2+r)t)^{\delta} (1+(1+r)(2+r)t)^{\frac{1}{2}}\big].
\end{gather}
This is the Enriques analogue of equation~\eqref{e1}, but the half integer exponents complicate our analysis.

To determine the sign of the Segre integral, we carry out the usual change of variables
\begin{gather*}
t=\frac{u}{1-(1+r)u}.
\end{gather*}
{\samepage Then we rewrite \eqref{pro} as
\begin{align*}
\int_{X^{[k]}} s_{2k}\big(F^{[k]}\big) &= \operatorname{Res}_{t=0} (1\!+(1\!+r)t)^{\chi-\delta-k(r+1) - \frac{1}{2}} (1\!+(2\!+r)t)^{\delta} (1\!+(1\!+r)(2\!+r)t)^{\frac{1}{2}}\frac{\mathrm dt}{t^{k+1}}
\\
&=\operatorname{Res}_{u=0} (1-(1+r)u)^{-\chi+k(r+2)-1}(1+u)^{\delta} \big(1+(1+r)^2u\big)^{\frac{1}{2}} \frac{\mathrm du}{u^{k+1}}
\\
&=\operatorname{Coeff}_{u^k} (1-(1+r)u)^{-\alpha} (1+u)^{\delta}\big(1+(1+r)^2u\big)^{\frac{1}{2}},
\end{align*}
where $\alpha=\chi-k(r+2)+1\geq kr+1$.}

We note that when $r$ is odd, $\delta\geq 0$ is an integer. Thus, the middle term $(1+u)^{\delta}$ has nonnegative coefficients and constant term equal to $1$. We claim that
\begin{gather*}
(1-(1+r)u)^{-\alpha}\big(1+(1+r)^2u\big)^{\frac{1}{2}}
\end{gather*}
has positive coefficients up to order $k$. The same will therefore be true after multiplying by $(1+u)^{\delta}$, showing that the top Segre class is positive.

A different idea is needed to establish the above claim. We change variables $u\mapsto u/(1+r)$ and consider instead the series
\begin{gather*}
W=(1-u)^{-\alpha}(1+(1+r)u)^{\frac{1}{2}}.
\end{gather*}
For each $0\leq m\leq k$, the coefficient of $u^m$ in $W$ equals
 \begin{equation}\label{eee}
 \sum_{i+j=m} \binom {-\alpha}{i} \binom{\frac{1}{2}}{j} (-1)^{i} (1+r)^j=\sum_{i+j=m} \binom{\alpha+i-1}{i} \binom{\frac{1}{2}}{j} (1+r)^j.
 \end{equation}
 The term corresponding to $i=m$, $j=0$ is clearly positive since $\alpha\geq 1$. We will ignore this term for the analysis. The next term $i=m-1$, $j=1$ is positive as well. The remaining terms however have alternating signs because of the fractional binomials. We show nonetheless that the sum of the consecutive $(i, j)$ and $(i-1, j+1)$ terms is positive, for $j$ odd:
\begin{equation}\label{summ}
\binom {\alpha+i-1}{i} \binom{\frac{1}{2}}{j} (1+r)^j+\binom {\alpha+i-2}{i-1} \binom{\frac{1}{2}}{j+1} (1+r)^{j+1}>0.
\end{equation}
This proves that the alternating sum \eqref{eee} is positive as well. To justify \eqref{summ}, we note that
\begin{gather*}
\binom {\alpha+i-1}{i}=\binom {\alpha+i-2}{i-1} \frac{\alpha+i-1}{i}>0, \qquad \binom{\frac{1}{2}}{j+1} = \binom{\frac{1}{2}}{j} \frac{\frac{1}{2}-j}{j+1}.
\end{gather*}
For $j$ odd, we have $\binom{\frac{1}{2}}{j} >0$. After cancellation, the inequality to establish becomes
\begin{gather*}
\frac{\alpha+i-1}{i}+(1+r) \frac{\frac{1}{2}-j}{j+1}>0.
\end{gather*}
Writing $i=m-j$, and using $\alpha-1=\chi-(r+2)k\geq rk\geq rm$, it suffices to establish
\begin{gather*}
\frac{rm+m-j}{m-j}+(1+r) \frac{\frac{1}{2}-j}{j+1}>0\iff j^2r+\frac{r+3}{2} (m-j)+mr>0,
\end{gather*} which is clearly true.
\end{proof}

\begin{Corollary}
If $L$ is a $(k-1)$-very ample line bundle over an Enriques surface $X$, $k\geq 2$, then $L^{[k]}$ is big and nef.

In particular, if $H$ is an ample line bundle over an Enriques surface, and $L_n=H^{\otimes n}$ then $(L_n)^{[k]}$ is big and nef for all $n\geq k+1$.
\end{Corollary}

\begin{proof} The second half of the corollary follows from the first. Indeed, it is noted in \cite[Proposition $2.5$]{S} that if $H$ is ample then $L_n=H^{\otimes n}$ is $(k-1)$-very ample for all $n\geq k+1$.

We prove the first statement. If $\chi:=\chi(L)\geq 4k$, the assertion follows from Theorem \ref{tt2} with $r=1$. Now, over Enriques surfaces, $(k-1)$-very ampleness imposes numerical restrictions on $L$ which are stronger than for the other $K$-trivial surfaces. Indeed, it was noted in \cite[Theorem~2.4]{S} that if $L$ is $(k-1)$-very ample, then $L^2\geq (k+1)^2$.
When $k\geq 6$, this is sufficient to guarantee that
\begin{gather*}
\chi(L)=1+\frac{L^2}{2}\geq 1+\frac{(k+1)^2}{2}>4k.
\end{gather*}
When $2\leq k\leq 5$, the bound $\chi\geq 4k$ required by Theorem \ref{tt2} may fail. However, the finitely many cases
\begin{gather*}
4k>\chi\geq 1+ \frac{(1+k)^2}{2}, \qquad 2\leq k\leq 5,
\end{gather*}
can be checked by hand using equation \eqref{pro}.
\end{proof}

\section{Other geometries} \label{secother}
The Segre integrals for arbitrary surfaces are established only when $F$ has rank $1$ or $2$, see \cite{MOP1, MOP}. In rank $1$, the answers were conjectured by Lehn \cite{Le}. The formulation below can be found in \cite{MOP}:
\begin{equation}\label{lehnc}
\sum_{n=0}^{\infty} z^k \int_{X^{[k]}} s\big(L^{[k]}\big) = A_1 (z)^{L^2} A_2(z)^{\chi(\mathcal O_X)} A_3 (z)^{L.K_X} A_4(z)^{K_X^2}.
\end{equation}
Here, for $z=t(1+2t)^2$, we have
\begin{gather*}
A_1(z)=(1+2t)^{\frac{1}{2}},
\\
A_2(z)=(1+2t)^{\frac{3}{2}} (1+6t)^{-\frac{1}{2}},
\\
A_3(z)=\frac{1}{2} (1+2t)^{-1} \left(\sqrt{1+2t}+\sqrt{1+6t}\right),
\\
 A_4(z)=4 (1+2t)^{\frac{1}{2}} (1+6t)^{\frac{1}{2}} \left(\sqrt{1+2t}+\sqrt{1+6t}\right)^{-2}.
 \end{gather*}

\subsection{A positivity result} It is more difficult to determine the sign of the top Segre integral from the above formulas. The~following lemma plays an important role in the analysis. We will apply it in the next section to geometric situations.

\begin{Lemma}\label{cl}
For all integers $m$, $n$, $p$ with $m\geq 0$, $p\geq 0$ such that $m+n+p$ is even, the series
\begin{gather*}
f(t)=\left(\sqrt{1+2t}+\sqrt{1+6t}\right)^m {(1+2t)}^{\frac{n-1}{2}} (1+6t)^{\frac{p-1}{2}}
\end{gather*}
has positive coefficients up to order less or equal than $\min\big(\frac{1}{2}(m+n+p)-1, m-1\big)$.
\end{Lemma}
Taking the minimum is necessary. Indeed, for $(m, n, p)=(2, 19, 1)$ the term of order $\frac{1}{2}(m+n+p)-1$ has negative coefficient, while for $(m, n, p)=(4, 0, 0)$, the term of order $m-1$ has zero coefficient. The lemma also holds for $m+n+p$ odd, but this case {\it never} occurs geometrically. The hypothesis $m, p\geq 0$ can fail in geometric examples, but our proof requires it.

\begin{proof}
The argument is not straightforward due to the alternating signs
in the expansions
\begin{gather*}
\sqrt{1+2t}= 1+t - \frac{t^2}{2}+\frac{t^3}{2}-\frac{5t^4}{8}+\cdots,
\\
\sqrt{1+6t}=1 + 3t - \frac{9t^2}{2}+\frac{27t^3}{2} - \frac{405 t^4}{8} +\cdots.
\end{gather*}
However, a residue calculation and suitable changes of variables will render the answer manifestly positive.

We begin by considering the case $n\geq 0$. In this situation, assume first $m$, $n$, $p$ are even, and~set
\begin{gather*}
F(t)=\left(\sqrt{1+2t}+\sqrt{1+6t}\right)^m {(1+2t)}^{-\frac{1}{2}} (1+6t)^{-\frac{1}{2}}.
\end{gather*}

\begin{Claim}\label{claim}
The series $F$ has positive coefficients up to order less or equal than $\frac{m}{2}-1$, and no terms of order between $\frac{m}{2}$ and $m-1$.
\end{Claim}

We assume this for now. Note that
\begin{gather*}
f(t)=F(t) P(t), \quad P(t)=(1+2t)^{\frac{n}{2}} (1+6t)^{\frac{p}{2}}.
\end{gather*}
Clearly, $P$ is a polynomial of degree $\frac{n+p}{2}$ with positive coefficients. This observation together with Claim~\ref{claim} gives the argument. Indeed, let $k\leq \min\big(\frac{1}{2}(m+n+p)-1, m-1\big)$. Writing $F_i$ and $P_j$ for the coefficients of $F$ and $P$, we have
\begin{gather*}
\operatorname{Coeff}_{ t^k} f(t)=\sum F_iP_j, \qquad\text{the sum ranging over}\quad i+j=k, \quad i\geq 0,\quad 0\leq j\leq \frac{n+p}{2}.
\end{gather*}
The sum is nonnegative since $i\leq k\leq m-1$, so $F_i\geq 0$ by the Claim~\ref{claim}, and $P_j> 0$. The sum is in fact strictly positive. Indeed, for $k\geq \frac{n+p}{2}$, it contains the term
\begin{gather*}
F_{k-\frac{n+p}{2}} P_{\frac{n+p}{2}}>0.
\end{gather*}
The last statement is also a consequence of Claim~\ref{claim}, using that $k-\frac{n+p}{2}\leq\frac{m}{2}-1$. When $k<\frac{n+p}{2}$, the sum contains the term $F_0 P_k=2^m P_k>0$.

\begin{proof}[Proof of Claim~\ref{claim}]
We seek to show that for all $k\leq \frac{m}{2}-1$, the residue
\begin{gather*}
\operatorname{Res}_{t=0} \frac{F(t)}{t^{k+1}} \,{\mathrm dt}>0.
\end{gather*}
The peculiar change of variables
\begin{gather*}
t=\frac{s(s+1)}{2\sqrt{3}s+(2+\sqrt{3})}
\end{gather*}
will simplify the calculation.
For convenience, set
\begin{gather*}
a=\frac{2-\sqrt{3}}{2\sqrt{3}}>0\qquad\text{so that}\quad t=\frac{s(s+1)}{2\sqrt{3}(s+a+1)}.
\end{gather*}
We note the following identities
\begin{gather*}
\sqrt{1+2t}=\frac{s+\frac{\sqrt{3}+1}{2}}{3^{\frac{1}{4}} (s+a+1)^{\frac{1}{2}}}, \qquad \sqrt{1+6t}=\frac{3^{\frac{1}{4}}\big(s+\frac{\sqrt{3}+1}{2\sqrt{3}}\big)}{(s+a+1)^{\frac{1}{2}}},
\\
\sqrt{1+2t}+\sqrt{1+6t}=\frac{(\sqrt{3}+1)(s+1)}{3^{\frac{1}{4}}(s+a+1)^{\frac{1}{2}}},\qquad
{\mathrm dt}=\frac{\big(s+\frac{\sqrt{3}+1}{2}\big)\big(s+\frac{\sqrt{3}+1}{2\sqrt{3}}\big)}{2\sqrt{3} (s+a+1)^2}\,{\mathrm ds}.
\end{gather*}
From here, we obtain
\begin{gather*}
\operatorname{Res}_{t=0} \frac{F(t)}{t^{k+1}} \,{\mathrm dt}=\operatorname{Res}_{s=0} \frac{G(s)}{s^{k+1}} \,{\mathrm ds},
\end{gather*}
where
\begin{gather*}
G(s)=\alpha (s+1)^{m-k-1} (s+a+1)^{k-\frac{m}{2}}, \qquad
\alpha=2^k \big(\sqrt{3}+1\big)^{m} \sqrt{3}^{k-\frac{m}{2}}>0.
\end{gather*}
A second change of variables will be needed next (this change of variables could have been carried out simultaneously with the first, but this is a price worth paying for readability). We~set
\begin{gather*}
s=\frac{(a+1)^2}{a} \frac{u}{1-\frac{a+1}{a}u}.
\end{gather*}
The reader can verify that
\begin{gather*}
s+1=\frac{1+(a+1)u}{1-\frac{a+1}{a}u}, \qquad s+a+1=\frac{a+1}{1-\frac{a+1}{a}u}, \qquad
{\mathrm ds}=\frac{(a+1)^2}{a} \frac{1}{\big(1-\frac{a+1}{a}u\big)^2}\,{\mathrm du}.
\end{gather*}
By direct calculation, we find
\begin{gather*}
\operatorname{Res}_{s=0} \frac{G(s)}{s^{k+1}}\,{\mathrm ds}=\operatorname{Res}_{u=0} \frac{H(u)}{u^{k+1}}\,{\mathrm du}
\end{gather*}
for
\begin{gather*}
H(u)=\beta (1+(a+1)u)^{m-k-1} \biggl(1-\frac{a+1}{a}u\biggr)^{k-\frac{m}{2}}, \qquad \beta=\alpha a^{k} (a+1)^{-k-\frac{m}{2}}>0.
\end{gather*}
The first term is a polynomial with positive coefficients $a_i$ given by binomial numbers, for $0\leq i\leq m-k-1$. The second term also has positive coefficients
\begin{gather*}
b_j=\biggl(-\frac{a+1}{a}\biggr)^j\binom{k-\frac{m}{2}}{j}>0
\end{gather*}
since $k-\frac{m}{2}<0$. Thus
\begin{gather*}
\operatorname{Coeff}_{u^k} H(u)=\!\sum a_i b_j,\quad \text{the sum ranging over}\!\quad i+j=k, \quad 0\leq i\leq m-k-1, \quad j\geq 0.
\end{gather*}
This coefficient is positive since $a_i>0$, $b_j> 0$ and the sum is non-empty (the term $a_0b_k=b_k>0$ appears in the sum).

When $m$ is even, and $\frac{m}{2}\leq k\leq m-1$, the expression $H(u)$ is a polynomial in $u$ of degree $\frac{m}{2}-1<k$. Hence, the coefficient of $u^k$ in $H(u)$ vanishes.
This proves the claim (and completes the argument when $m$, $n$, $p$ are even and $n\geq 0$).
\end{proof}

When $n\geq 0$, there are three other cases to consider, which require different choices for $F$ and $P$:
\begin{itemize}\itemsep=0pt
\item [$(i)$] when $m$ even, $n$, $p$ odd, we set
\begin{gather*}
F(t)=\left(\sqrt{1+2t}+\sqrt{1+6t}\right)^m, \qquad
P(t)=(1+2t)^{\frac{n-1}{2}} (1+6t)^{\frac{p-1}{2}},
\end{gather*}
\item [$(ii)$] when $m$ odd, $n$ even, $p$ odd, we set
\begin{gather*}
F(t)=\left(\sqrt{1+2t}+\sqrt{1+6t}\right)^m {(1+2t)}^{-\frac{1}{2}}, \qquad
P(t)=(1+2t)^{\frac{n}{2}} (1+6t)^{\frac{p-1}{2}},
\end{gather*}
\item [$(iii)$] when $m$ odd, $n$ odd, $p$ even, we set
\begin{gather*}
F(t)=\left(\sqrt{1+2t}+\sqrt{1+6t}\right)^m (1+6t)^{-\frac{1}{2}}, \qquad
P(t)=(1+2t)^{\frac{n-1}{2}} (1+6t)^{\frac{p}{2}}.
\end{gather*}
\end{itemize}
In case $(i)$, $F$ has positive coefficients up to order $\frac{m}{2}$, and no terms of order between $\frac{m}{2}+1$ and $m-1$.
In cases $(ii)$ and~$(iii)$, $F$ has positive coefficients up to order $\frac{m-1}{2}$, and no terms of order between $\frac{m+1}{2}$ and $m-1$. The arguments are similar, and we leave the details to the reader.

When $n<0$, the reasoning used to prove Claim~\ref{claim} also works to deal directly with the function
\begin{gather*}
f(t)=\left(\sqrt{1+2t}+\sqrt{1+6t}\right)^m {(1+2t)}^{\frac{n-1}{2}} (1+6t)^{\frac{p-1}{2}}.
\end{gather*}
Following exactly the same steps, first changing from $t$ to $s$ and then from $s$ to $u$, we obtain
\begin{align*}
H(u)={}&\gamma (1+(a+1)u)^{m-k-1} \biggl(1-\frac{a+1}{a}u\biggr)^{k-\frac{m+n+p}{2}} \biggl(1-\frac{(\sqrt{3}+1)^3}{4\sqrt{3}}u\biggr)^{n}
\\
&\times \biggl(1+\frac{\big(\sqrt{3}+1\big)^3}{4}u\biggr)^p,
\end{align*}
for $\gamma>0$. By the same arguments, this has positive coefficients when the exponents
\begin{gather*}
m-k-1\geq 0, \qquad k-\frac{m+n+p}{2}<0, \qquad n<0, \qquad p\geq 0,
\end{gather*}
which we assumed.
\end{proof}

\subsection[K3 blowups]{$\boldsymbol{K3}$ blowups}
The top Segre classes computed by formula \eqref{lehnc} are always coefficients in series of the form~$f(t)$ as in Lemma \ref{cl}, barring the condition $m, p\geq 0$. When this condition is satisfied, we easily obtain big and nef criteria for the tautological bundles $L^{[k]}\to X^{[k]}$.

There are several specific examples where our techniques apply. We illustrate them first when
\begin{gather*}
\pi\colon\ X\to S
\end{gather*}
is the blowup of a $K3$ surface $S$ at a point $p\in S$. We assume
$S$ has Picard rank $1$, with ample Picard generator $H$. Let $H^2=2h$.

\begin{Theorem}\label{blok}
Let $E$ be the exceptional divisor on $X$, and set $L=H-\ell E$. Assume $\ell\geq k-1$ and
\begin{gather*}
2h>\max \bigl((\ell+2)^2-6, (\ell+1)^2+4k, \ell(\ell+1)+6k-6\bigr).
\end{gather*}
Then $L^{[k]}$ is big and nef on $X^{[k]}$.
\end{Theorem}

\begin{proof} We first show $H-(\ell+1) E$ nef. Recall the Sheshadri constant
\begin{gather*}
\epsilon (S, p)= \max \{t \in \mathbb R_{\geq 0}\colon H-tE \text{ is nef}\},
\end{gather*}
see for instance \cite[Definition 5.1.1]{L-I}. Note that since $H$ is in the nef cone of $X$, if $H-tE$ is nef, then $H-t'E$ is also nef for all $0\leq t'\leq t$. Thus it suffices to explain that
\begin{gather*}
\epsilon(S, p)\geq \ell+1.
\end{gather*}
The Sheshadri constants of $K3$ surfaces of Picard rank $1$ have been studied in \cite{K} and shown to satisfy
\begin{gather*}
\epsilon(S, p)\geq \big\lfloor {\sqrt{H^2}}\big\rfloor,
\end{gather*}
with two possible exceptions
\begin{gather*}
H^2=\alpha^2+\alpha-2,\qquad \epsilon(S, p)\geq \alpha - \frac{2}{\alpha+1}
\end{gather*}
and
\begin{gather*}
H^2=\alpha^2+\frac{\alpha-1}{2},\qquad \epsilon(S, p)\geq \alpha -\frac{1}{2\alpha+1}, \qquad \alpha\in \mathbb Z_{>0}.
\end{gather*}
Since the inequality $2h>\ell(\ell+3)$ is implied by our hypothesis, it follows that $\epsilon(S, p)\geq \ell+1$ in all cases (using $\alpha\geq \ell+2$ in the two exceptional cases).

We next show that $L$ is $(k-1)$-very ample. Note that $K_X=E$, and write
\begin{gather*}
L=K_X+M, \qquad M=H-(\ell+1)E.
\end{gather*}
Observe that $M$ is nef by the first paragraph of the proof, and
\begin{gather*}
M^2=2h-(\ell+1)^2>4k.
\end{gather*}
By \cite[Theorem 2.1]{BS}, if $L$ is not $(k-1)$-very ample, there exists an effective divisor $D\neq 0$ such that
\begin{gather*}
D.M-k\leq D^2<\frac{D.M}{2}<k.
\end{gather*}
Furthermore, $M-2D$ is $\mathbb Q$-effective and $D$ contains a subscheme $\zeta$ of length at most
equal to~$k$ such that
\begin{gather*}
H^0(L)\to H^0(L\otimes \mathcal O_\zeta)
\end{gather*}
is not surjective. Write \begin{gather*}D=aH+bE,
\end{gather*}
and note that $D$ effective implies $a\geq 0$. Similarly,
\begin{gather*}
M-2D=(1-2a)H+(-\ell-1-2b)E
\end{gather*}
is $\mathbb Q$-effective, so $1-2a\geq 0$. Thus $a=0$, $D=bE$ with $b > 0$. We have
\begin{gather*}
\frac{D.M}{2}<k\implies b(\ell+1)<2k\implies b<\frac{2k}{\ell+1}\leq 2
\end{gather*}
since $k\leq \ell+1$. Hence $b=1$ and $D=E$. For subschemes $\zeta$ of $E$, the map
\begin{gather*}
H^0(L)\to H^0(L\otimes \mathcal O_\zeta)
\end{gather*}
can be written as composition
\begin{gather*}
H^0(L)\to H^0(L|_{E}), \qquad H^0(L|_E)\to H^0(L\otimes \mathcal O_{\zeta}).
\end{gather*}
Since $L|_{E}=\mathcal O_E(\ell)$ and $\zeta$ has length less or equal to $k\leq \ell+1$, the second map is clearly surjective. The first map is also surjective since
\begin{gather*}
H^1(L(-E))=0.
\end{gather*}
This is a consequence of \cite[Proposition 4.1]{V} and requires the bound $2h>(\ell+2)^2-6$, which we assumed. We conclude $L^{[k]}$ is globally generated and thus nef.

It remains to explain that the top Segre class of $L^{[k]}$ is positive. We use \eqref{lehnc} and we change variables
\begin{gather*}
z=t(1+2t)^2\qquad \text{so that}\quad \mathrm dz=(1+2t)(1+6t)\,{\mathrm dt}.
\end{gather*}
Then
\begin{align}
\int_{X^{[k]}} s\big(L^{[k]}\big)&=\operatorname{Res}_{z=0} A_1^{2h-\ell^2} A_2^2 A_3^{\ell} A_4^{-1} \frac{\mathrm dz}{z^{k+1}}\nonumber
\\
\nonumber&=
\operatorname{Res}_{t=0} 2^{-\ell-2} (1+2t)^{h-\frac{\ell^2}{2}-2k-\ell+\frac{3}{2}} (1+6t)^{-\frac{1}{2}} \left(\sqrt{1+2t}+\sqrt{1+6t}\right)^{\ell+2}\frac{\mathrm dt}{t^{k+1}}
\\
&= 2^{-\ell-2} \operatorname{Coeff}_{t^k} (1\!+2t)^{h-\frac{\ell^2}{2}-2k-\ell+\frac{3}{2}} (1\!+6t)^{-\frac{1}{2}} \left(\sqrt{1+2t}\!+\sqrt{1+6t}\right)^{\ell+2}\!.
\label{cvvv}
\end{align}
We are now in the situation considered in Lemma \ref{cl}. Thus, the coefficient above is positive provided $k\leq \ell+1$ and
\begin{gather*}
k\leq \frac{1}{2} \bigg((\ell+2)+2\bigg(h-\frac{\ell^2}{2}-2k-\ell+2\bigg)\bigg)-1\iff 2h> \ell(\ell+1)+6k-6.
\end{gather*}
This completes the proof.
\end{proof}

\subsection{Surfaces of general type} It is natural to wonder how far these techniques take us. We show

\begin{Theorem} \label{gtyp}Assume $X$ is a smooth projective minimal surface of general type. Let $L$ be a~$(k-1)$-very ample line bundle such that
\begin{gather*}
\chi(L)\geq 3k, \qquad L.K_X\geq 2K_X^2+k+1.
\end{gather*}
Then $L^{[k]}$ is big and nef over $X^{[k]}$.
\end{Theorem}

\begin{proof}
Arguing as in the proof of Theorem \ref{blok}, in particular equation \eqref{cvvv}, we express the Segre integral as the $t^k$-coefficient in the series
\begin{gather*}
f(t)=2^{-m} \left(\sqrt{1+2t}+\sqrt{1+6t}\right)^m {(1+2t)}^{\frac{n-1}{2}} (1+6t)^{\frac{p-1}{2}},
\end{gather*}
where
\begin{gather*}
m=L.K-2K^2, \qquad n=(L-K)^2+3\chi - 4k - 1, \qquad p=K^2-\chi+3,
\end{gather*}
with $K=K_X$ and $\chi=\chi(\mathcal O_X)$. Note that
\begin{gather*}
m+n+p=2\chi(L)-4k+2
\end{gather*}
is even. Furthermore, $p\geq 0$. Indeed, let $p_g$, $q$ denote the genus and irregularity of $X$. By~Noet\-her's inequality $K^2\geq 2p_g-4$, and $K^2>0$ since $X$ is minimal of general type. Averaging, we obtain $K^2>p_g-2$, and thus
\begin{gather*}
p=K^2-\chi+3>(p_g-2)-(1-q+p_g)+3=q\geq 0.
\end{gather*}
By Lemma \ref{cl}, the coefficients of
$f(t)$ up to order $\min\big(\frac{1}{2} (m+n+p)-1, m-1\big)$ are positive. In~particular, if
\begin{gather*}
k\leq \min\bigg(\frac{1}{2} (m+n+p)-1, m-1\bigg),
\end{gather*}
then the Segre integral is positive. The latter inequality is exactly our hypothesis, and thus $L^{[k]}$ is big and nef.
\end{proof}

\section{Curves} \label{seccur} Let $C$ be a smooth projective curve of genus $g$, and let $V\to C$ be a vector bundle with $\chi=\chi(V)$. It is natural to ask whether the previous results also apply to the tautological bundles $V^{[k]}\to C^{[k]}$. In fact, by similar considerations, we establish an analogue of Theorem~\ref{tt}:

\begin{Proposition} \label{tt3}Assume $V\to C$ is a $(k-1)$-very ample rank $r$ vector bundle with $\chi \geq (r+1)k$.
Then $V^{[k]}\to C^{[k]}$ is big and nef over $C^{[k]}$.
\end{Proposition}

To illustrate, if $V$ is stable of degree $d>r(2g-2+k)$, then $V$ is $(k-1)$-very ample. This holds because for any divisor $Z\subset C$ of degree $k$, we have
\begin{gather*}
H^1(V(-Z))=H^0\big(V^{\vee}\otimes K_C(Z)\big)=0
\end{gather*}
by Serre duality, stability and the assumption $\mu(V)>2g-2+k$.

\begin{proof}
Since $V$ is $(k-1)$-very ample, it follows that $V^{[k]}$ is globally generated, hence nef. To~prove $V^{[k]}$ is big, it suffices to verify
\begin{gather*}
(-1)^k \int_{C^{[k]}} s\big(V^{[k]}\big)>0.
\end{gather*}
The latter integrals are computed in \cite[Theorem 2]{MOP1}; they can be viewed as higher rank analogues of the classical $k$-secant integrals. The answer bears analogies with equation \eqref{eq1}:
\begin{gather*}
\sum_{k=0}^{\infty} z^k \int_{C^{[k]}} s_{k} \big(V^{[k]}\big) =A_1^{d} A_2^{1-g},
\end{gather*}
where
\begin{gather*}
A_1(z)=1+t, \qquad A_2(z)=\frac{(1+t)^{r+1}}{1+(1+r)t}, \qquad z=-t(1+t)^r.
\end{gather*}
We can express the Segre integrals as residues
\begin{align*}
(-1)^k \int_{C^{[k]}} s_{k} \big(V^{[k]}\big)&=(-1)^k \operatorname{Res}_{z=0} A_1^d A_2^{1-g} \frac{\mathrm dz}{z^{k+1}}
\\
&=\operatorname{Res}_{t=0} (1+t)^{\chi - kr - g} (1+(1+r)t)^{g} \frac{\mathrm dt}{t^{k+1}}.
\end{align*}
Just as in Theorem \ref{tt}, we further change variables $t=\frac{u}{1-u}$, so that
\begin{align*}
(-1)^k \int_{C^{[k]}} s_{k} \big(V^{[k]}\big)&=\operatorname{Res}_{u=0} (1+ru)^{g} (1-u)^{-\chi+k(r+1)-1} \frac{\mathrm du}{u^{k+1}}
\\
&= \operatorname{Coeff}_{u^k} (1+ru)^{g} (1-u)^{-\chi+k(r+1)-1}.
\end{align*}
By the same reasoning as in Theorem~\ref{tt}, we see that this coefficient is positive when $-\chi+k(r+1)-1<0\iff \chi\geq k(r+1)$. This completes the argument.
\end{proof}

To go further, we consider the punctual Quot schemes $\mathsf {Quot}_C\big(\mathbb C^N, k\big)$ parametrizing quotients
\begin{gather*}
0\to S\to \mathbb C^N\otimes \mathcal O_C\to Q\to 0, \qquad \rk Q=0, \qquad \text{length }Q=k.
\end{gather*}
These are smooth projective varieties of dimension $Nk$, and carry tautological vector bundles~$V^{[k]}$ for each vector bundle $V\to C$:
\begin{gather*}
V^{[k]}=p_{\star} (\mathcal Q\otimes q^{\star} V).
\end{gather*}
Here $\mathcal Q$ is the universal quotient and $p$, $q$ are the two projections over $\mathsf{Quot}_C\big(\mathbb C^N, k\big)\times C$.
The associated Segre integrals were studied in~\cite{OP2021}. We extend the above results to the punctual Quot scheme, in rank $1$. The higher rank case appears more involved.

\begin{Theorem} \label{quotpun}
Let $L$ be a line bundle with $\chi (L)\geq k+g$ and $\chi(L)\geq k\big(1+\frac{1}{N}\big)$. Then, the vector bundle $L^{[k]}$ is big and nef over $\mathsf {Quot}_C\big(\mathbb C^N, k\big)$.
\end{Theorem}

\begin{proof} We show that $L^{[k]}$ is globally generated, hence nef. The universal sequence over $\mathsf {Quot}_C\big(\mathbb C^N, k\big)\times C$
\begin{gather*}
0\to \mathcal S\to \mathbb C^N\otimes \mathcal O\to \mathcal Q\to 0
\end{gather*}
induces, via tensorization by $L$ followed by pushforward, a morphism
\begin{gather*}
H^0\bigl(C, L^{\oplus N}\bigr)\otimes \mathcal O_{\mathsf{Quot}}\to L^{[k]}.
\end{gather*}
To prove that the morphism is surjective, we establish that
\begin{equation}\label{ss}
H^0\bigl(C, L^{\oplus N}\bigr)\to H^0(C, L\otimes Q)
\end{equation}
is surjective for all length $k$ punctual quotients $Q$ of the rank $N$ trivial bundle.

The argument only requires $L$ to be $(k-1)$-very ample, which is certainly true for us. In~fact, $L$ is $(k-1)$-very ample whenever $\deg L\geq 2g-1+k\iff \chi(L)\geq g+k$, as we remarked before the proof of Proposition~\ref{tt3}. Surjectivity of~\eqref{ss} for $N=1$ is a rephrasing of $(k-1)$-very ampleness.

For the general case, we induct on $N$. We pick a splitting $\mathbb C^N=\mathbb C\oplus \mathbb C^{N-1}$, and form the diagram with exact rows and columns:
\begin{center}
$\xymatrix{
& \ar[d] 0 & \ar[d]0 & \ar[d]0 &\\
0\ar[r] & S' \ar[r]\ar[d] & S\ar[r]\ar[d] & S''\ar[r]\ar[d] & 0 \\
0\ar[r] & \mathcal O_C \ar[r]\ar[d]& \mathbb C^{N}\otimes \mathcal O_{C}\ar[r]\ar[d] &\mathbb C^{N-1}\otimes \mathcal O_{C} \ar[r]\ar[d]& 0 \\
0\ar[r] & Q' \ar[r]\ar[d] & Q\ar[r]\ar[d] & Q''\ar[r]\ar[d] & 0\\
& 0 & 0 & 0 & }$
\end{center}
Tensoring with $L$ and taking global sections yields
\begin{center}
$\xymatrix{
 0\ar[r] & \ar[r]\ar[d]H^0(C, L)& \ar[r]\ar[d]H^0\big(C, L^{\oplus N}\big) & \ar[r] \ar[d]H^0\big(C, L^{\oplus (N-1)}\big)\ar[d] &0\\
0 \ar[r]& H^0(C, L\otimes Q')\ar[r]\ar[d] & H^0(C, L\otimes Q)\ar[r]\ & H^0(C, L\otimes Q'')\ar[r] \ar[d] & 0 \\
 & 0 & & 0 & }$
\end{center}

Both $Q'$, $Q''$ have lengths less or equal to $k$, the length of $Q$. Since $L$ is $(k-1)$-very ample, it is also $(\ell-1)$-very ample for all $1\leq \ell\leq k$, in particular when $\ell$ equals the length of $Q'$ or~$Q''$. Thus, the first and last vertical arrows are surjective by induction. This implies that the middle vertical arrow is surjective as well.

It remains to show that $L^{[k]}$ is big. By \eqref{pos}, it suffices to determine the sign of the top Segre class of $L^{[k]}$. No additional calculation is needed in this case. Indeed, the Segre integrals were noted in \cite[Corollary~10]{OP2021} to satisfy the symmetry
\begin{gather*}
(-1)^{Nk} \int_{{\mathsf{Quot}}_C(\mathbb C^N, k)} s\big(L^{[k]}\big)=(-1)^k\int_{C^{[k]}} s\bigl(\big(L^{\oplus N}\big)^{[k]}\bigr).
\end{gather*}
In the proof of Proposition \ref{tt3}, the latter integral was shown to be positive if $N\chi (L)\geq (N+1)k$, which is true by hypothesis.
\end{proof}

\subsection*{Acknowledgements}

We are grateful to G.~Bini, S.~Boissi\`ere, F.~Flamini for correspondence related to~\cite{BBF}; their paper served as motivation for this work. We thank A.~Marian and R.~Pandharipande for collaboration that led to \cite{MOP2, MOP1, MOP, OP2021}. We thank the referees for their careful reading of the manuscript and for their comments. The author is supported by NSF grant DMS1802228.

\pdfbookmark[1]{References}{ref}
\LastPageEnding

\end{document}